\documentclass[12pt]{amsart}
\usepackage{amssymb,amsmath,amsthm,comment}
\usepackage{pdfpages,graphicx} 
\usepackage[section]{placeins} 
\usepackage{wrapfig}
\usepackage{float}
\usepackage{lineno}
\usepackage{youngtab}

\usepackage{setspace}
\newcommand{\shrinkmargins}[1]{
  \addtolength{\textheight}{#1\topmargin}
  \addtolength{\textheight}{#1\topmargin}
  \addtolength{\textwidth}{#1\oddsidemargin}
  \addtolength{\textwidth}{#1\evensidemargin}
  \addtolength{\topmargin}{-#1\topmargin}
  \addtolength{\oddsidemargin}{-#1\oddsidemargin}
  \addtolength{\evensidemargin}{-#1\evensidemargin}
  }
\shrinkmargins{0.5}

\newtheorem{theorem}{Theorem}
\newtheorem{lemma}[theorem]{Lemma}

\newtheorem{corollary}[theorem]{Corollary}
\newtheorem*{theorem*}{Theorem}

\newtheorem{proposition}[theorem]{Proposition}

\theoremstyle{definition}
\newtheorem{definition}{Definition}
\newtheorem{example}{Example}

\theoremstyle{remark}
\newtheorem*{remark}{Remark}

\numberwithin{theorem}{section} \numberwithin{equation}{section}

\usepackage{multirow}

\newcommand{\floor}[1]{\left\lfloor #1 \right\rfloor}

\begin{document}
\title[Infinite Products]{Log-Concavity and Log-Convexity of Restricted Infinite Products}
\author{Krystian Gajdzica}
\address{Theoretical Computer Science Department, Faculty of Mathematics and Computer Science, Jagiellonian University, Łojasiewicza 6, 30--348 Kraków, Poland}
\email{krystian.gajdzica@uj.edu.pl}
\author{Bernhard Heim}
\address{Department of Mathematics and Computer Science\\Division of Mathematics\\University of Cologne\\ Weyertal 86--90 \\ 50931 Cologne \\Germany}
\email{bheim@uni-koeln.de}
\author{Markus Neuhauser}
\address{Kutaisi International University, 5/7, Youth Avenue,  Kutaisi, 4600 Georgia}
\email{markus.neuhauser@kiu.edu.ge}
\address{Lehrstuhl f\"{u}r Geometrie und Analysis, RWTH Aachen University, 52056 Aachen, Germany}
\subjclass[2020] {Primary 05A17, 11P82; Secondary 05A20}
\keywords{Generating functions, log-concavity,
log-convexity, partition numbers, symmetric group.}
\begin{abstract}In this paper we provide a  classification 
on the sign distribution of $\Delta _{E,\ell}(n):=
p_{E,\ell }(n)^2 - p_{E,\ell }(n-1) \, p_{E,\ell }(n+1)$, where
\begin{equation*}
\sum_{n =0}^{\infty}
p_{E,\ell }(n) \, q^n := \prod_{n \in S}
 \left(1 - q^n \right)^{-f_{\ell}(n)},\quad (\ell \in \mathbb{N}, f_1\equiv 1).
\end{equation*}
We take the product over $1\in
S \subset \mathbb{N}$ and denote the complement by
$E$, the set of exceptions. In the case of $\ell=1$ and $E$
the multiples of $k$, $p_{E,1}\left( n\right) $ represents the number of $k$-regular partitions. More generally, let $f_{\ell}$ satisfy a certain growth condition. We determine the signs of $\Delta _{E,\ell }(n)$ for $\ell$ large. The signs
mainly depend on
the occurrence
of subsets of $\{2,3,4,5\}$ as a part of the exception set and the residue class of $n$ modulo $
r$, where $r
$ depends on $E$. For example, let $2,3 \in S$ and 
$4$ an exception. Let $n$ be large. Then
for almost all $\ell$ we have
\begin{equation*}
\Delta _{E,\ell }(n) >0 \,\,\, \text{ for } n\equiv 2 \pmod{3}.
\end{equation*}
If we assume $3,4 \in S$ and $2$ an  exception.
Let $n$ be large. 
Then
for almost all $\ell$ we have
\begin{equation*}
\Delta _{E,\ell }(n) < 0 \,\,\, \text{ for } n\equiv 2 \pmod{3}.
\end{equation*}
Note that this property is independent of the integers
$k\in S,k>4$.
\end{abstract}

\maketitle
\newpage

\section{Introduction and
main
results}

Meinardus \cite{Me54} demonstrated and proved that it is possible to
get precise results in studying the asymptotic behavior of the
$q$-expansion of general infinite products of Eulerian type.
He obtained asymptotic formulas of such partition numbers $p_f(n)$ to weight functions $f$
(we also refer to \cite{An98}, Chapter 6). 
Debruyne and Tenenbaum \cite{DT20}
extended Meinardus' results to partitions with restrictions on the parts
with support $\mathbb{N} \setminus
f^{-1}\left( \left\{ 0\right\} \right) $. Granovsky, Stark, and Erlihson
introduced a probabilistic method to extend Meinardus' result due to Khintchine
which applies to a variety of models in physics and combinatorics 
\cite{GSE08, GS06,GS12, GS15}.

Recently, Bridges, Brindle, Bringmann, and Franke \cite{BBBF24} proved asymptotic formulas if the support is a multiset of integers and the zeta function has multiple poles. 
This had been successfully applied \cite{BFH24} to certain weight functions $f_{\ell}(n)$ 
\begin{equation*}\label{weights}
\sum_{n=0}^{\infty} p_{\ell}(n) \, q^n = \prod_{n=1
}^{\infty} \left( 1 - q^n \right)^{-f_{\ell}(n)}
\end{equation*}
(e.~g.\
$n^{\ell}$ and
$\lambda_{\ell}$, where $\lambda_{\ell}(n)$ counts the number of subgroups of $\mathbb{Z}^{\ell}$ of index $n$).

In this paper we 
consider sequences of weight functions $\{f_{\ell}\}_{\ell}$
and invest in the log-concavity and log-convexity 
of the partition functions $p_{E,
{\ell }}\left( n\right) $ and its 
stability with respect to the configuration of the exceptional set $E$.
We restrict the infinite product to $n \in S \subset \mathbb{N}$ and
denote by
$E:= \mathbb{N} \setminus S $ the exception set.
We assume that $1 \in S$. 

To state our results, we introduce some notation.
\begin{definition}\label{main:def}
Throughout the paper, let $\{ f_{\ell}\}_{\ell}$ be a sequence of normalized arithmetic functions. We assume that $f_{\ell}: \mathbb{N} \longrightarrow \mathbb{N}$ and that
the initial function $f_1 \equiv 1$.
Further, let
real valued functions $\Phi
\left( {\ell}\right) , \Psi
\left( {\ell}\right)
$ on $\mathbb{N}$ exist such that
the following hold true.
\begin{enumerate}
\item  For all $\ell \in \mathbb{N}$ we have
$n^{\Phi(\ell)} \leq f_{\ell}(n) \leq 
n^{\Psi(\ell)}
$.

\item  There exists a positive number $B
$ independent of $\ell$ such that $ 0 \leq \Psi(\ell) - \Phi(\ell) \leq B
$.
\item  We have $\lim_{\ell \to \infty} \Phi(\ell) = \infty$.
\end{enumerate}

\end{definition}
Let $S$ be a subset of $\mathbb{N}$ with $1 \in S$. Then we
have the double sequence
$\{p_{E,\ell }(n)\}_{\ell,n}$ of partition numbers by
\begin{equation*}\label{exception}
\sum_{n=0}^{\infty} p_{E,\ell }(n) \, q^n = \prod_{n \in S}
\left( 1 - q^n \right)^{-f_{\ell}(n)}.
\end{equation*}

Moreover, we define
\begin{equation*}
\Delta _{E,\ell }(n):= \left(p_{E,\ell }(n) \right)^2 - 
p_{E,\ell }(n-1)\,  p_{E,\ell }(n+1) .
\end{equation*}
The double sequence
$\{p_{E,\ell }(n)\}_{\ell,n}$ is log-concave at $n$ for the parameter $\ell$ if $\Delta _{E,\ell }\left( n\right)
\geq 0$ and strictly log-concave if $\Delta _{E,\ell }(n) > 0$. 
We have the same definition for log-convexity if $\Delta _{E,\ell }(n) \leq 0$, and
strict log-convexity if $\Delta _{E,\ell }(n) < 0$. 
\newline

We begin by demonstrating the scope of our results on generalized 
$k$-regular partitions. Let $f_{\ell}(n):= n^{\ell -1}$ and 
$E_k:=\{k^m\, : \, m \in \mathbb{N} \}$
for $k>1$. Let $P$ denote the set of all partitions of $n$.
We denote by
$p_{E_k,\ell}(n)$ the generalized $k$-regular partitions, since the number of elements of $p_{E_k,1}(n)$ is given by
\begin{equation*}
\Big\vert \left\{ \lambda=(\lambda_1, \lambda_2, \ldots, \lambda_d )\in P \, : \, \lambda_j \neq k^m, \, d,m \in \mathbb{N} \text{ and } 1 \leq j \leq d \right\} \Big\vert.
\end{equation*}
Further, $p_{E_k,2}(n)$ is related to restricted plane partitions.
Note, that for $E=\emptyset$ (no exceptions) and $n\geq 6$ that $\Delta_{\emptyset,\ell}(n)$
is positive for almost all $\ell$ and $n\equiv 0 \pmod{3}$. Further,
$\Delta_{\emptyset,\ell}(n)$ is negative for almost all $\ell$ if $n \equiv 1,2 \pmod{3}$ (we refer to \cite{HN22}).

Let $n \geq 4$ and $k=2$. Then we prove that $\Delta_{E_2,\ell}$ is positive 
for almost all $\ell$ if $n$ is divisible by $3$ and negative for almost all $\ell$ otherwise. This follows from our Theorem \ref{th13} and Theorem \ref{th135}. 
Note, the same is also true if $1,3,5 \in S$ and $2,4 \in E$,
which shows already the dominance of $2$ and $4$. We also refer to Figure \ref{24}.
\begin{center}
\begin{figure}[!htb]\label{24}
   \begin{minipage}{1\textwidth}
     \centering
     \includegraphics[width=0.7\linewidth]{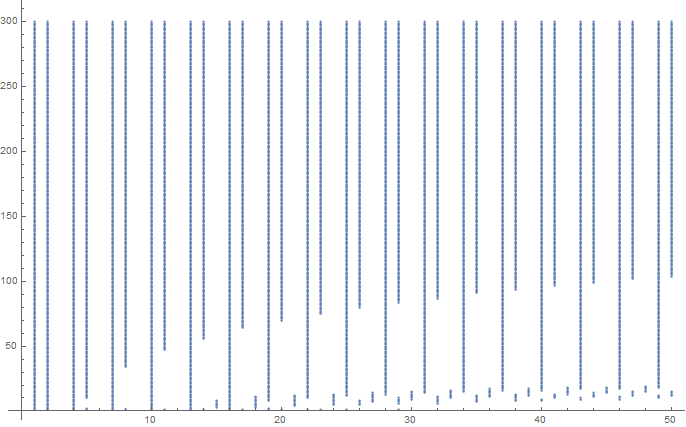}
     \caption{All the pairs $(n,\ell)$ such that $\Delta_{E,\ell}(n)\leq 0$, where
     $f_{\ell}(n)=n^{\ell -1}$, $E=\{2,4\}$, $1\leq\ell\leq300$ and $1\leq n\leq 50$. }
   \end{minipage}\hfill
\end{figure}
\end{center}

In contrast, let $1,2,3 \in S$ and $4 \in E$ (but not $2$). Then we have for every $n \geq 3$, that $\Delta_{E,\ell}$ is positive for almost all $\ell$ if $n \not\equiv 1 \pmod{3}$ and negative for almost all $\ell$ if $n \equiv 1 \pmod{3}$. 
This follows from Theorem \ref{th123} and Theorem \ref{th1234}.

Now we consider what is happening in the case of generalized $3$-regular partitions. Let $n\geq 5$. Then we have for $\ell$ large dependent on $n$:
\begin{eqnarray*}
\Delta _{E_3,\ell }(n) >
0 \text{ for }n\equiv 0 \pmod{2} \text{ and }
\Delta _{E_3,\ell }(n) <
0 \text{ for }n\equiv 1 \pmod{2}. 
\end{eqnarray*}
This follows from our Theorem \ref{th12}. Note, the same
is true if we only request that $1,2\in S$ and $3 \in E$.

Finally, let $k \geq 5$. Then again by Theorem \ref{th123} we obtain
that for almost all $n$ we have $\Delta_{E_k,\ell}(n) >0$ for $n \equiv 0 \pmod{3}$ and $\Delta_{E_k,\ell}(n) <0$ for $n \equiv 1 \pmod{3}$. 
The case $n\equiv 2 \pmod{3}$ remains mysterious (we refer to Example 1 and Example 2).

\section{Classification}
\subsection{The case $1,2,3 \in S $}
We need to distinguish between the cases related to $n \not\equiv 2 \pmod{3}$ and
$n \equiv 2 \pmod{3}$.

\begin{theorem}
\label{th123}Let $\{f_{\ell}\}_{\ell}$ be given as in Definition \ref{main:def}.
Then we have for all $n\geq3$ with $n\not\equiv 2 \pmod{3}$ the following log-concavity and log-convexity properties. Let $1,2,3$
not be
exceptions. Then we have for $\ell$ large dependent on $n$:
\begin{eqnarray*}
\Delta _{E,\ell }\left( n\right) >
0 \text{ for }n
\equiv 0 \pmod{3} \text{ and }\Delta _{E,\ell}\left( n\right)
<0 \text{ for }n\equiv 1
\pmod{3}. 
\end{eqnarray*}
\end{theorem}
The proof of this Theorem is given in Section \ref{subsection S=(1,2,3)}.
\begin{theorem}
\label{th1234}Let $\{f_{\ell}\}_{\ell}$ be given as in Definition \ref{main:def}.
Then we have for all $n\geq3$ with $n \equiv 2 \pmod{3}$ the following log-concavity and log-convexity properties. Let $1,2,3$
not be
exceptions. Then we have for $\ell$ large dependent on $n$:
Let $4 \in E$. Then $\Delta _{E,\ell }\left( n\right)
> 0$ for almost all $\ell$.
\end{theorem}
The proof of this Theorem is given in Section 
\ref{Paragraph: 1,2,3 not in E 4 in E}.
\begin{example}
Let $E=\emptyset $,
$f_{\ell }\left( n\right) =n^{\ell -1}$ for $n\neq 2$ and
$f_{\ell }\left( 2\right) =2^{\ell }$. Then for
$\Phi \left( \ell \right) =\ell -1$ and
$\Psi \left( \ell \right) =\ell $ obviously
$n^{\Phi \left( \ell \right) }\leq f_{\ell }\left( n\right) \leq
n^{\Psi \left( \ell \right) }$.
Then $g_{\ell }\left( 2\right) =1+2^{\ell +1}$,
$g_{\ell }\left( 3\right) =1+3^{\ell }$, and
$g_{\ell }\left( 4\right) =1+2^{\ell +1}+4^{\ell }$.

For $n>1$ fixed, we obtain the following asymptotic expansions in $\ell $
\[
p_{E,\ell }\left( n\right) \sim \left\{
\begin{array}{ll}
\frac{1}{\left( \frac{n}{3}\right) !}3^{\frac{n}{3}\left( \ell -1\right) }
,&n\equiv 0\pmod{3},\\
\frac{3}{\left( \frac{n-4}{3}\right) !}
\left( 4\cdot 3^{\frac{n-4}{3}}\right) ^{\ell -1}
,&n\equiv 1\pmod{3},\\
\frac{2
}{\left( \frac{n-2}{3}\right) !}
\left( 2\cdot 3^{\frac{n-2}{3}}\right) ^{\ell -1}
,&n\equiv 2\pmod{3}.
\end{array}
\right.
\]
Therefore, for $n\equiv 2\pmod{3}$ and $n>2$, we obtain
\[
\Delta _{E,\ell }\left( n\right)
\sim \frac{n+10
}{3\left( \frac{n-2}{3}\right) !\left( \frac{n+1}{3}\right) !}
\left( 2\cdot 3^{\frac{n-2}{3}}\right) ^{2\ell -2}
>0
\]
for sufficiently large $\ell $.
\end{example}

\begin{example}
Let $E=\emptyset $. Let $f_{\ell }\left( n\right) =n^{\ell }$ for
$n\notin \left\{ 2,4\right\} $,
$f_{\ell }\left( 2\right) =2^{\ell +\left( -1\right) ^{\ell }}$, and
$f_{\ell }\left( 4\right) =4^{\ell -\left( -1\right) ^{\ell }}$.
Obviously $n^{\ell -1}\leq f_{\ell }\left( n\right) \leq n^{\ell +1}$
so we can choose $\Phi \left( \ell \right) =\ell -1$ and
$\Psi \left( \ell \right) =\ell +1$. Then we obtain
$g_{\ell }\left( 2\right)
\sim 2^{1+\ell +\left( -1\right) ^{\ell }}$,
$g_{\ell }\left( 3\right)
\sim 3^{1+\ell }$, and
$g_{\ell }\left( 4\right)
\sim 4^{1+\ell -\left( -1\right) ^{\ell }}$.
Furthermore,
\[
p_{\ell }\left( n\right) \sim \beta _{n}\left( \ell \right) =\left\{
\begin{array}{ll}
\frac{1}{\left( \frac{n}{3}\right) !}3^{\ell \frac{n}{3}},&n\equiv 0\pmod{3},\\
\frac{1}{\left( \frac{n-4}{3}\right) !}
\left( \frac{1}{2}4^{\ell +\left( -1\right) ^{\ell }}
+4^{\ell -\left( -1\right) ^{\ell }}\right) 3^{\ell \frac{n-4}{3}},
&n\equiv 1\pmod{3},\\
\frac{1}{\left( \frac{n-2}{3}\right) !}
2^{\ell +\left( -1\right) ^{\ell }}3^{\ell \frac{n-2}{3}},&n\equiv 2\pmod{3}.
\end{array}
\right.
\]
For $n\equiv 2\pmod{3}$, $n>2$ we obtain
\begin{eqnarray*}
&&\left( \beta _{n}\left( \ell \right) \right) ^{2}
-\beta _{n-1}\left( \ell \right) \beta _{n+1}\left( \ell \right) \\
&=&\left( \frac{1}{\left( \frac{n-2}{3}\right) !}
2^{\ell +\left( -1\right) ^{\ell }}3^{\ell \frac{n-2}{3}}\right) ^{2}
-\frac{1}{\left( \frac{n+1}{3}\right) !\left( \frac{n-5}{3}\right) !}\left(
\frac{1}{2}4^{\left( -1\right) ^{\ell }}+4^{-\left( -1\right) ^{\ell }}
\right) 4^{\ell }3^{\ell \frac{2n-4}{3}}\\
&=&\frac{1}{\left( \frac{n+1}{3}\right) !\left( \frac{n-2}{3}\right) !}
\left( \frac{n+4}{6}4^{\left( -1\right) ^{\ell }}
-\frac{n-2}{3}4^{-\left( -1\right) ^{\ell }}\right)
4^{\ell }3^{\ell \frac{2n-4}{3}}.
\end{eqnarray*}
Therefore,
\[
\Delta _{E,\ell }\left( n\right)
\sim \frac{1}{\left( \frac{n+1}{3}\right) !\left( \frac{n-2}{3}\right) !}
4^{\ell }3^{\ell \frac{2n-4}{3}}\left\{
\begin{array}{ll}
\frac{7n+34}{12}>0, & \ell \equiv 0\pmod{2}, \\
\frac{-31n+68}{24}<0, & \ell \equiv 1\pmod{2}. 
\end{array}
\right.
\]
\end{example}

\subsection{The case $1,2 \in S $ and $3 \in E $}
This case is very interesting, since the assumption $2 \in S$ and $3 \in E$ already
determines the sign of $\Delta _{E,\ell }(n)$ for large $\ell$ dependent on only $n \pmod{2}$ for $n \geq 5$. The proof of the 
following Theorem \ref{th12} is provided in Section \ref{subsection42}
and Section \ref{subsection43}.
\begin{theorem}
\label{th12}Let $\{f_{\ell}\}_{\ell}$ be given as in Definition \ref{main:def}.
Then we have for all $n\geq 5$ the following log-concavity and log-convexity properties. Let
$3 \in E$ be an exception and let $1,2 \in S$
not be exceptions. Then we have for $\ell$ large dependent on $n$:
\begin{eqnarray*}
\Delta _{E,\ell }(n) >
0 \text{ for }n\equiv 0 \pmod{2} \text{ and }
\Delta _{E,\ell }(n) <
0 \text{ for }n\equiv 1 \pmod{2}. 
\end{eqnarray*}
\end{theorem}
\subsection{The case $1,3 \in S $ and $2,4 \in E$}
We need to distinguish between the cases related to $n \not\equiv 1 \pmod{3}$ and
$n \equiv 1 \pmod{3}$.

\begin{theorem}
\label{th13}Let $\{f_{\ell}\}_{\ell}$ be given as in Definition \ref{main:def}.
Then we have for all $n\geq 4$ with $n\not\equiv 1 \pmod{3}$ the following log-concavity and log-convexity properties. Let $1,3 \in S$ and $2,4 \in E$. Then we have for $\ell$ large dependent on $n$:
\begin{eqnarray*}
\Delta _{E,\ell }(n) >
0 \text{ for }n\equiv 0 \pmod{3} \text{ and }
\Delta_{E,\ell}(n) <
0 \text{ for }n\equiv 2 \pmod{3}. 
\end{eqnarray*}
\end{theorem}
The proof of this Theorem \ref{th13} and the next Theorem \ref{th135} is provided in Section \ref{subsection44},  Section \ref{subsection S=(1,3)}.
\begin{theorem}
\label{th135}Let $\{f_{\ell}\}_{\ell}$ be given as in Definition \ref{main:def}.
Then we have for all $n \geq 4$ with $n \equiv 1 \pmod{3}$ the following log-concavity and log-convexity properties. Let $1,3 \in S$ and $2,4 \in E$. Then we have for $\ell$ large dependent on $n$ and $5 \in S$ that
$\Delta _{E,\ell }\left( n\right) <0$ for almost all $\ell$.
\end{theorem}
We leave the remaining possibility $5 \in E$
as an open challenge to the reader.

\begin{example}
\label{example13}

\begin{enumerate}
\item  Let $S=\left\{ 1,3\right\}
$.
We obtain
\begin{eqnarray*}
\sum _{n=0}^{\infty }
p_{E,\ell }\left( n\right) q^{n}
&=&
\left( 1-q\right) ^{-1}
\left( 1-q^{3}\right)
^{-f_{\ell }\left( 3\right) }
=\left( 1+q+q^{2}\right)
\left( 1-q^{3}\right)
^{-f_{\ell }\left( 3\right) -1}\\
&=&\left( 1+q+q^{2}\right)
\sum _{n=0}^{\infty }
\binom{n
+f_{\ell }\left( 3\right) }{n}q^{3n}.
\end{eqnarray*}
This shows that
$p_{E,\ell }\left( 3n\right)
=p_{E,\ell }\left( 3n+1\right)
=p_{E,\ell }\left( 3n+2\right)$.
This shows
$\Delta _{E,\ell }\left( n\right)
=0$
for $n\equiv 1\pmod{3}$.

\item  Now let $n\geq6$, $1,3,7
\in S$ but
$2,4,5,6,8,9\in E$.
Then the
maximizing partitions are
\begin{equation}
\begin{array}{rl}
\left( 3^{\frac{n}{3}}\right) ,
&n\equiv 0\pmod{3},\\
\left( 3^{\frac{n-1}{3}},
1^{1}\right) ,
&n\equiv 1\pmod{3},\\
\left( 3^{\frac{n-2}{3}},
1^{2}\right) ,
&n\equiv 2\pmod{3}.
\end{array}
\label{max13}
\end{equation}
The corresponding partitions
yielding the second largest
products are
\[
\begin{array}{rl}
\left( 3^{\frac{n-
3}{3}},
1^{3}\right) ,
&n\equiv 0\pmod{3},\\
\left( 7^{1},
3^{\frac{n-7}{3}}\right) ,
&n\equiv 1\pmod{3},\\\left( 7^{1},3^{\frac{n-8}{3}},1^{1}\right) ,&n\equiv 2\pmod{3}.
\end{array}
\]
In the following we already utilize the arithmetic function $g_{E,\ell}(n)$, which is introduced in (\ref{defg}).

For the leading and subleading
growth terms of $p_{E,\ell }\left( n\right) $
for sufficiently large $\ell $
we obtain
\begin{eqnarray*}
&&p_{E,\ell }\left( n\right) \sim \beta _{n}\left( \ell \right) \\
&=&\left\{
\begin{array}{ll}
\frac{1}{\left( \frac{n}{3}\right) !}
\left( \frac{g_{E,\ell }\left( 3\right) }{3}\right) ^{\frac{n}{3}}
+\frac{2}{3}\frac{1}{\left( \frac{n-3}{3}\right) !}
\left( \frac{g_{E,\ell }\left( 3\right) }{3}\right) ^{\frac{n-3}{3}},
&n\equiv 0\pmod{3},\\
\frac{1}{\left( \frac{n-1}{3}\right) !}
\left( \frac{g_{E,\ell }\left( 3\right) }{3}\right) ^{\frac{n-1}{3}}
+\frac{1}{\left( \frac{n-7}{3}\right) !}
\frac{g_{E,\ell }\left( 7\right) }{7}
\left( \frac{g_{E,\ell }\left( 3\right) }{3}\right) ^{\frac{n-7}{3}},
&n\equiv 1\pmod{3},\\
\frac{1}{\left( \frac{n-2}{3}\right) !}
\left( \frac{g_{E,\ell }\left( 3\right) }{3}\right) ^{\frac{n-2}{3}}
+\frac{1}{\left( \frac{n-8}{3}\right) !}
\frac{g_{E,\ell }\left( 7\right) }{7}
\left( \frac{g_{E,\ell }\left( 3\right) }{3}\right) ^{\frac{n-8}{3}},
&n\equiv 2\pmod{3}.
\end{array}
\right.
\end{eqnarray*}
For $n\equiv 1\pmod{3}$
we obtain
\begin{eqnarray*}
&&\left( \beta _{n}\left( \ell \right) \right) ^{2}
-\beta _{n-1}\left( \ell \right) \beta _{n+1}\left( \ell \right) \\
&=&\left( \frac{g_{E,\ell }\left( n\right) }{3}\right) ^{\frac{2n-14}{3}}
( \frac{1}{\left( \frac{n-1}{3}\right) !\left( \frac{n-7}{3}\right) !}
\frac{g_{E,\ell }\left( 7\right) }{7}
\left( \frac{g_{E,\ell }\left( 3\right) }{3}\right) ^{2}
+\left( \frac{1}{\left( \frac{n-7}{3}\right) !}
\frac{g_{E,\ell }\left( 7\right) }{7}\right) ^{2}\\
&&{}-\frac{2}{3}
\frac{1}{\left( \frac{n-1}{3}\right) !\left( \frac{n-4}{3}\right) !}
\left( \frac{g_{E,\ell }\left( 3\right) }{3}\right) ^{3}
-\frac{2}{3}
\frac{1}{\left( \frac{n-4}{3}\right) !\left( \frac{n-7}{3}\right) !}
\frac{g_{E,\ell }\left( 7\right) }{7}\frac{g_{E,\ell }\left( 3\right) }{3}
).
\end{eqnarray*}
By (\ref{eq:fundamental}) the leading growth term of that difference is determined by
\[
\frac{g_{E,\ell }\left( 7\right) }{7}
\left( \frac{g_{E,\ell }\left( 3\right) }{3}\right) ^{2}\geq
63^{\Phi \left( \ell \right) }
\]
we obtain
\[
\Delta _{E,\ell }\left( n\right) \sim
\frac{1}{\left( \frac{n-1}{3}\right) !\left( \frac{n-7}{3}\right) !}
\frac{g_{E,\ell }\left( 7\right) }{3}
\left( \frac{g_{E,\ell }\left( 3\right) }{3}\right) ^{\frac{2
n-8}{3}}>0.
\]

\item  Let now $n\geq6$, $1,3,8\in S$ but $2,4,5,6,7,9,10\in E$. The maximizing partitions
are as in (\ref{max13}) but the second largest are in this case
\[
\begin{array}{rl}
\left( 3^{\frac{n-3}{3}},1^{3}\right) , & n\equiv 0\pmod{3}, \\
\left( 3^{\frac{n-4}{3}},1^{4}\right) , & n\equiv 1\pmod{3}, \\
\left( 8^{1},3^{\frac{n-8}{3}}\right) , & n\equiv 2\pmod{3}.
\end{array}
\]
Therefore we obtain for the leading and subleading growth terms
\begin{eqnarray*}
&&p_{E,\ell }\left( n\right) \sim \beta _{n}\left( \ell \right) \\ &=&\left\{
\begin{array}{ll}
\frac{1}{\left( \frac{n}{3}\right) !}
\left( \frac{g_{E,\ell }\left( 3\right) }{3}\right) ^{\frac{n}{3}}
+\frac{2}{3}\frac{1}{\left( \frac{n-3}{3}\right) !}
\left( \frac{g_{E,\ell }\left( 3\right) }{3}\right) ^{\frac{n-3}{3}},
&n\equiv 0\pmod{3},\\
\frac{1}{\left( \frac{n-1}{3}\right) !}
\left( \frac{g_{E,\ell }\left( 3\right) }{3}\right) ^{\frac{n-1}{3}}
+\frac{2}{3}\frac{1}{\left( \frac{n-4}{3}\right) !}
\left( \frac{g_{E,\ell }\left( 3\right) }{3}\right) ^{\frac{n-4}{3}},
&n\equiv 1\pmod{3},\\
\frac{1}{\left( \frac{n-2}{3}\right) !}
\left( \frac{g_{E,\ell }\left( 3\right) }{3}\right) ^{\frac{n-2}{3}}
+\frac{1}{\left( \frac{n-8}{3}\right) !}\frac{g_{E,\ell }\left( 8\right) }{8}
\left( \frac{g_{E,\ell }\left( 3\right) }{3}\right) ^{\frac{n-8}{3}},
&n\equiv 2\pmod{3}.
\end{array}
\right.
\end{eqnarray*}
Let now $n\equiv 1\pmod{3}$. Then we obtain
\begin{eqnarray*}
&&\left( \beta _{n}\left( \ell \right) \right) ^{2}
-\beta _{n-1}\left( \ell \right) \beta _{n+1}\left( \ell \right) \\
&=&\left( \frac{g_{E,\ell }\left( 3\right) }{3}\right) ^{\frac{2n-11}{3}}
( \frac{2}{3}
\frac{1}{\left( \frac{n-1}{3}\right) !
\left( \frac{n-4}{3}\right) !}
\left( \frac{g_{E,\ell }\left( 3\right) }{3}\right) ^{2}
+\left( \frac{2}{3}\frac{1}{\left( \frac{n-4}{3}\right) !}\right) ^{2}
\frac{g_{E,\ell }\left( 3\right) }{3}\\
&&{}-\frac{1}{\left( \frac{n-1}{3}\right) !\left( \frac{n-7}{3}\right) !}
\frac{g_{E,\ell }\left( 8\right) }{8}\frac{g_{E,\ell }\left( 3\right) }{3}
-\frac{2}{3}
\frac{1}{\left( \frac{n-4}{3}\right) !\left( \frac{n-7}{3}\right) !}
\frac{g_{E,\ell }\left( 8\right) }{8}
).
\end{eqnarray*}
Again by (\ref{eq:fundamental}) the leading growth term of this difference is determined by
$\frac{g_{E,\ell }\left( 8\right) }{8}\frac{g_{E,\ell }\left( 3\right) }{3}
\geq 24^{\Phi \left( \ell \right) }$
we obtain
$\Delta _{E,\ell }\left( n\right) \sim
-\frac{1}{\left( \frac{n-1}{3}\right) !\left( \frac{n-7}{3}\right) !}
\frac{g_{E,\ell }\left( 8\right) }{8}\left(
\frac{g_{E,\ell }\left( 3\right) }{3}\right) ^{\frac{2n-8}{3}}
<0$.

\end{enumerate}
\end{example}

\subsection{The
case $1,r,r+1 \in S$ for $r\geq 3$ and $2,3 \ldots r-1 \in E$}

It is quite surprising that the following result has no subcases.
\begin{theorem}
\label{th1rr+1}Let $\{f_{\ell}\}_{\ell}$ be given as in Definition \ref{main:def}.
Let $r \geq 3$ and $1,r,r+1 \in S$ for $r\geq 3$ and 
$2,3 \ldots r-1 \in E$. Then for
\begin{equation*}
n \geq \frac{r(r-1)(3r-1)+4}{2}
\end{equation*}
we have $\Delta _{E,\ell }\left( n\right) >
0$ if
$n \not \equiv r-1 \pmod{r}$
for almost all $\ell$.
And $\Delta _{E,\ell }\left( n\right)
< 0$ in the remaining case 
$n\equiv r-1 \pmod{r}$ for almost all $\ell$.
\end{theorem}

This is proven in Section~\ref{subsection S=(1,r,r+1)} and
Section~\ref{problematic subsection S=(1,r,r+1)}.

\section{Preliminaries}\label{Preliminaries}
Let $\{f_{\ell}\}_{\ell}$ be given as in Definition \ref{main:def}.
We associate
a sequence $\{g_{E,\ell }\}_{\ell}$ of functions on $\mathbb{N}$
via the following well-known formula
\begin{equation}\label{exp}
\sum_{n=0}^{\infty} p_{E,\ell }(n) \, q^n = 
\prod_{n \in
{S}}
\left( 1 - q^n \right)^{-f_{\ell}(n)} = 
\exp \left( \sum_{n=1}^{\infty} g_{E,\ell }(n) \frac{q^n}{n}\right).
\end{equation}
We have
\begin{equation}
g_{E,\ell }(n):= \sum_{\substack{d \mid n \\ d \notin E}} d \, f_{\ell}(d).
\label{defg}
\end{equation}
For example, let $f_{\ell}(n)= n^{\ell -1}$ and an exception set $E$ be given.
Then $g_{E,\ell }(n) =\sum_{\substack{d \mid n \\ d \notin E}}
d^{\ell}$ which we denote by
$\sigma_{E,\ell }(n)$. If $\ell =1$ we usually drop the index and as usual we have
$\sigma_{\ell}(n) = \sigma_{\emptyset,\ell }(n)$ if we have no exceptions.
Since $1$ is no exception we have ensured that the largest divisor
$n_{{S}}$
of $n$ in $S$ exists. The following bounds for $g_{E,\ell }(n)$
will turn out to be very useful.
\begin{lemma} Let $E$ and $\{f_{\ell}\}_{\ell}$ be given. Then
\begin{equation}\label{eq:fundamental}
\left( n_{S}\right)^{\Phi(\ell)+1} \leq g_{E,\ell }(n) 
\leq \sigma_{E,1}(n) \, \left(n_{S}\right)^{\Psi(\ell)}
\end{equation}
\end{lemma}

\begin{proof}
We have
$g_{E,\ell }\left( n\right)
=\sum _{\substack{d\mid n \\ d\notin E}}df_{\ell }\left( d\right)
\geq n_{S}f_{\ell }\left( n_{S}\right)
\geq n_{S}^{\Phi \left(
\ell \right) +1}$
and
$g_{E,\ell }\left( n\right)
\leq \sum _{\substack{d\mid n \\ d\notin E}}dn_{S}^{\Psi \left( \ell \right) }
\leq \sigma _{E,1}\left( n\right) n_{S}^{\Psi \left( \ell \right) }$.
\end{proof}


From the second equation of (\ref{exp}) and (\ref{eq:fundamental}), we get that
\begin{eqnarray*}
p_{E,\ell }\left( n\right) & =  & 
\sum _{k\leq n}\sum _{\substack{m_{1},\ldots ,m_{k} \geq 1 \\ m_{1}+\ldots +m_{k}=n}}
\frac{1}{k!} \, \frac{g_{E,\ell }\left( m_{1}\right) 
\cdots g_{E,\ell }\left( m_{k}\right) }{m_{1}\cdots m_{k}} \\
&\leq &
\sum _{k\leq n}\sum _{\substack{m_{1},\ldots ,m_{k} \geq 1 \\ m_{1}+\ldots +m_{k}=n}}
\frac{1}{k!} \, 
\frac{\sigma_{E,1}\left( m_{1}\right) 
\cdots \sigma_{E,1}\left( m_{k}\right) }{m_{1}\cdots m_{k}}
\, \left(m_{1,S} \cdots m_{k,S}\right)^{\Psi(\ell)},
\end{eqnarray*}
where $m_{j,S}$ is the largest divisor of $m_j$ belonging to the support $S$
($S=\mathbb{N}\setminus E$). Now, let us set
\begin{equation*}\label{max:M E}
M_{E}\left( n\right) :=\max _{k\leq n}
\max _{\substack{m_1,\ldots,m_k \in S \\ m_{1}+\ldots +m_{k}=n}}m_{1}\cdots m_{k}.
\end{equation*}
If it is clear from the context, we will write $M_E$ instead of $M_E(n)$. From the above discussion, we obtain the upper bound of the form
\begin{equation*}
p_{E,\ell }(n) \leq 
\sum _{k\leq n}\sum _{\substack{m_{1},\ldots ,m_{k} \geq 1 \\ m_{1}+\ldots +m_{k}=n}}
\frac{1}{k!} \, 
\frac{\sigma_{E,1}\left( m_{1}\right) 
\cdots \sigma_{E,1}\left( m_{k}\right) }{m_{1}\cdots m_{k}}
\, M_{E}^{\Psi(\ell)} = p_E(n) \, \left(M_E(n)\right)^{\Psi(\ell)}.
\label{pupper}
\end{equation*}
Next, we derive the lower bound. It follows that
\begin{eqnarray*}
p_{E,\ell }(n) &\geq& 
\sum_{k=1}^n \frac{1}{k!}
\sum_{\substack{m_1,\ldots,m_k\in S\\ m_1+\cdots+m_k=n}}
\frac{g_{E,\ell }(m_1)\cdots g_{E,\ell }(m_k)}{m_1\cdots m_k} \nonumber \\
&\geq &
\sum_{k=1}^n \frac{1}{k!}
\sum_{\substack{m_1,\ldots,m_k\in S\\ m_1+\cdots+m_k=n}}
\left(m_1 \cdots m_k\right)^{\Phi(\ell)} \nonumber \\
&\geq& \sum_{k=1}^n\frac{1}{k!}\sum_{\substack{m_1,\ldots,m_k\in S \\
m_1+\cdots+m_k=n \\ m_1\cdots m_k=M_{E}(n)}}\left(M_E(n)\right)^{\Phi(\ell)}.
\label{plower}
\end{eqnarray*}
This leads us to the following general result.

\begin{proposition}\label{prop: estimations}
Let $n\geq 1$. Let $\{f_{\ell}\}_{\ell}$ be given as in Definition \ref{main:def}.
Let $E$ be the exception set.
Let $\Phi$ and $\Psi$ determine the growth of $f_{\ell}$ as given in Definition \ref{main:def}.
Then there exist positive
constants $C_{E}^{\downarrow}(n)$ and $C_{E}^{\uparrow}(n)$ 
dependent only on $E$ and $n$ such that:
\begin{equation*}
\frac{C_{E}^{\downarrow}(n) 
\, M_{E}(n)^{2 \Phi(\ell)}}{\left( M_{E}(n-1) \, M_{E}(n+1)\right)^{\Psi(\ell)}}
\leq \frac{p_{E,\ell }(n)^2}{p_{E,\ell }(n-1) \, p_{E,\ell }(n+1)} \leq 
\frac{C_{E}^{\uparrow}(n) \, M_{E}(n)^{2 \Psi(\ell)}}{\left( M_{E}(n-1) \, M_{E}(n+1)\right)^{\Phi(\ell)}}.
\end{equation*}
\end{proposition}

Throughout the manuscript, we assume that $\Psi(\ell)
-\Phi(\ell) \leq
B$, where $B
\geq 0$ is fixed and independent of $\ell$. That requirement together with Proposition \ref{prop: estimations} delivers an efficient
criterion for the log-concavity or the log-convexity of
$p_{E,\ell }$ at $n$ for all sufficiently large values of $\ell$. More precisely, it turns out that the key information
lies at the value of the quotient
\begin{equation*}
Q_{E}(n) :=\frac{M_{E}(n)^{2}}{M_{E}(n-1) \, M_{E}(n+1)}
\label{q}
\end{equation*}
as the following criterion exhibits.
\begin{theorem}\label{thm: basic}
Let $n\geq 1$, $E$ and 
$\{f_{\ell}\}_{\ell}$ be given as in Definition \ref{main:def}.
Then $p_{E,\ell }(n)$ is strictly
\begin{enumerate}
\item  log-concave at $n$ for all sufficiently large values of $\ell$ if $Q_{E}(n)>1$;
\item  log-convex at $n$ for all sufficiently large values of $\ell$ if $Q_{E}(n)<1$.
\end{enumerate}
\end{theorem}
\begin{proof}
Since both of the cases are very similar, we only prove the criterion for the log-concavity property. Thus, let us require that $Q_E(n)>1$. Proposition \ref{prop: estimations} asserts that
\begin{align*}
\frac{p_{E,\ell }(n)^2}{p_{E,\ell }(n-1) \, p_{E,\ell }(n+1)}&\geq \frac{C_{E}^{\downarrow}(n) 
\, M_{E}(n)^{2 \Phi(\ell)}}{\left( M_{E}(n-1) \, M_{E}(n+1)\right)^{\Psi(\ell)}}\\
&\geq
\frac{C_{E}^{\downarrow}(n)}{\left( M_{E}(n-1) \, M_{E}(n+1)\right)^{B
}}\left(\frac{M_{E}(n)^{2 }}{M_{E}(n-1) \, M_{E}(n+1)}\right)^{\Phi(\ell)}\\
&=\frac{C_{E}^{\downarrow}(n)}{\left( M_{E}(n-1) \, M_{E}(n+1)\right)^{
B}}Q_E(n)^{\Phi(\ell)}\xrightarrow[\ell \to \infty]{} \infty,
\end{align*} 
where the last line is the consequence of the facts that $\Phi(\ell)\xrightarrow[\ell \to \infty]{} \infty$ and $Q_E(n)>1$. This completes the proof.
\end{proof}

\begin{remark}
There is a natural question arising from the statement of Theorem \ref{thm: basic}, namely, what can one say about the log-behavior of $p_{E,\ell }(n)$ if $Q_E(n)=1.$ This situation is more intricate, and we
saw in Example \ref{example13} that both phenomena (log-concavity and log-convexity) might occur in such a setting. 
\end{remark}

\begin{remark}
If there is a prime number $p$ such that $p\mid M_E(n)^2 $ and $p\nmid M_E(n-1)M_E(n+1)$ or conversely, then $p_{E,\ell }(n)$ is either strictly log-concave or strictly log-convex for all sufficiently large values of $\ell$.
\end{remark}

\section{Varying the
quotient $Q_{E}(n)$}

For various choices of the exception set $E$, we now apply Theorem \ref{thm: basic} to describe the log-behavior of $p_{E,\ell }(n)$ at $n$ for all sufficiently large values of $\ell$. Therefore, we need to determine the quotients
$Q_{E}(n)$. Actually, it turns out that their values only depend on the
exception set $E$ and the residue class of $n \pmod{r}$, where $r$ is
directly related to $E$.

\subsection{The case of $\mathbf{1,2,3 \not\in E} $}
\label{subsection S=(1,2,3)}

Here, Lemma \ref{Lemma: a_2=2} maintains that  
$r=3$---that is because the largest products heavily depend on the number
of $3$s appearing in their corresponding compositions. For $n\geq3$, it also guarantees that
\[
M_{E}\left( n\right)
=\left\{
\begin{array}{ll}
3^{n/3}, & n\equiv 0\pmod{3}, \\
4 \cdot 3^{(n-4)/3}, & n\equiv 1\pmod{3}, \\
2 \cdot 3^{(n-2)/3}, & n\equiv 2\pmod{3}.
\end{array}
\right.
\]
In consequence, this leads to
\begin{equation*}
Q_{E}\left( n\right) = \left\{
\begin{array}{ll}
\frac{9}{8}, & n\equiv 0\pmod{3}, \\
\frac{8}{9}, & n\equiv 1\pmod{3}, \text{ and} \\
1, & n\equiv 2\pmod{3}.
\end{array}
\right.
\end{equation*}
Therefore, Theorem \ref{thm: basic} asserts that $p_{E,\ell }(n)$ is log-concave if $n\equiv0\pmod{3}$ and log-convex if $n\equiv1\pmod{3}$ for all but finitely many numbers $\ell$.
This proves Theorem~\ref{th123}.
On the other hand, the case of $n\equiv 2 \pmod{3}$ needs some additional investigation.

\subsection{The case of $\mathbf{1,2,5 \not\in E}$ and $\mathbf{3\in E}$}
\label{subsection42}
In such a setting, Lemma \ref{Lemma: a_2=2} points out that $r=4$ and
\[
M_{E}\left( n\right)
=\left\{
\begin{array}{ll}
2^{n/2}, & n\equiv 0\pmod{4}, \\
5 \cdot 2^{(n-5)/2}, & n\equiv 1\pmod{4}, \\
2^{n/2}, & n\equiv 2\pmod{4}, \\
5 \cdot 2^{(n-5)/2}, & n\equiv 3\pmod{4},
\end{array}
\right.
\]
for any $n\geq4$. This gives us that 
\begin{equation*}
Q_{E}\left( n\right) =\left\{
\begin{array}{ll}
\frac{2^5}{5^2}, & n\equiv 0\pmod{4}, \\
\frac{5^2}{2^5}, & n\equiv 1\pmod{4}, \\
\frac{2^5}{5^2}, & n\equiv 2\pmod{4}, \text{ and}\\
\frac{5^2}{2^5}, & n\equiv 3\pmod{4}.
\end{array}
\right.
\end{equation*}
Hence, Theorem \ref{thm: basic} maintains that $p_{E,\ell }(n)$ is log-concave if $n\equiv0\pmod{2}$ and log-convex if $n\equiv1\pmod{2}$ for all but finitely many numbers $\ell$, and this is a complete characterization in that case.
This proves Theorem~\ref{th12} in
the case $5\in S$.

\subsection{The case of $\mathbf{1,2 \not\in E}$ and $\mathbf{3,5 \in E}$}
\label{subsection43}
Once again, Lemma \ref{Lemma: a_2=2} guarantees that $r=4$ and
\[
M_{E}\left( n\right) =\left\{
\begin{array}{ll}
2^{n/2}, & n\equiv 0\pmod{4}, \\
2^{(n-1)/2}, & n\equiv 1\pmod{4}, \\
2^{n/2}, & n\equiv 2\pmod{4}, \\
2^{(n-1)/2}, & n\equiv 3\pmod{4}
\end{array}
\right.
\]
for any non-negative integer $n$. This leads to 
\begin{equation*}
Q_{E}\left( n\right) =\left\{
\begin{array}{ll}
2, & n\equiv 0\pmod{4}, \\
\frac{1}{2}, & n\equiv 1\pmod{4}, \\
2, & n\equiv 2\pmod{4}, \text{ and} \\
\frac{1}{2}, & n\equiv 3\pmod{4}.
\end{array}
\right.
\end{equation*}
Therefore, we conclude that $p_{E,\ell }(n)$ is log-concave if $n\equiv0\pmod{2}$ and log-convex if $n\equiv1\pmod{2}$ for all but finitely many numbers $\ell$, and this is a complete characterization in that case.
This finishes the proof of
Theorem~\ref{th12} with the
case $5\in E$. 

\subsection{The case of $\mathbf{1,3,5 \not\in E}$ and $\mathbf{2,4\in E}$}
\label{subsection44}
In order to determine the maximal product $M_E(n)$ here, we use Lemma \ref{Lemma: a_2>2 a_j>2a_2}. It says that each number $u\geq6$ can not occur as a part of the compositions corresponding to $M_E(n)$ for any $n\geq3$. Moreover, the number $5$ may appear at most once---otherwise it is always better to replace $(5,5)$ by $(3,3,3,1)$. Thus, we put $r=3$, and obtain
\[
M_{E}\left( n\right) =\left\{
\begin{array}{ll}
3^{n/3}, & n\equiv 0\pmod{3}, \\
3^{(n-1)/3}, & n\equiv 1\pmod{3}, \\
5\cdot3^{(n-5)/3}, & n\equiv 2\pmod{3}
\end{array}
\right.
\]
for any $n\geq3$.
This further leads to 
\begin{equation*}
Q_{E}\left( n\right) =\left\{
\begin{array}{ll}
\frac{3^2}{5}, & n\equiv 0\pmod{3}, \\
\frac{3}{5}, & n\equiv 1\pmod{3}, \text{ and} \\
\frac{5^2}{3^3}, & n\equiv 2\pmod{3}.
\end{array}
\right.
\end{equation*}
In conclusion, $p_{E,\ell }(n)$ is log-concave if $n\equiv0\pmod{3}$ and log-convex if $n\equiv1,2\pmod{3}$ for all but finitely many numbers $\ell$, and this is a complete characterization in that case.
This finishes the proofs of
Theorem~\ref{th13}
and Theorem~\ref{th135} in the
case $5\in S$.

\subsection{The case of $\mathbf{1,3 \not\in E}$ and $\mathbf{2,4,5\in E}$}\label{subsection S=(1,3)}
Here, the direct application of Corollary \ref{Corollary: a_2>2 a_3>2a_2} gives us $r=3$ and
\begin{equation}
M_{E}\left( n\right) =\left\{
\begin{array}{ll}
3^{n/3}, & n\equiv 0\pmod{3}, \\
3^{(n-1)/3}, & n\equiv 1\pmod{3}, \\
3^{(n-2)/3}, & n\equiv 2\pmod{3}
\end{array}
\right.
\end{equation}
for every non-negative integer $n$.
This, in consequence, leads to 
\begin{equation*}
Q_{E}\left( n\right) =\left\{
\begin{array}{ll}
3, & n\equiv 0\pmod{3}, \\
1, & n\equiv 1\pmod{3},\text{ and} \\
\frac{1}{3}, & n\equiv 2\pmod{3}.
\end{array}
\right.
\end{equation*}
Again, Theorem \ref{thm: basic} asserts that $p_{E,\ell }(n)$ is log-concave if $n\equiv0\pmod{3}$ and log-convex if $n\equiv2\pmod{3}$ for all but finitely many numbers $\ell$.
This finishes the proof of
Theorem~\ref{th13} in the case
$5\in E$.
On the other hand, the case of $n\equiv 1 \pmod{3}$ needs some additional investigation.

\subsection{The case of $\mathbf{s\geq3}$, $\mathbf{1,s,s+1 \not\in E}$ and $\mathbf{2,3,\ldots,s-1 \in E}$}\label{subsection S=(1,r,r+1)}
Here, Lemma \ref{Lemma: a_3=a_2+1} maintains that $r=s$ and
\begin{equation*}
M_{E}\left( n\right)
=
\left( r+1\right) ^{j}r^{\frac{n-j(r+1)}{r}},\qquad n\equiv j\pmod{r}
\end{equation*}
for all $n>r(r-1)(3r-1)/2$ and $j\in\{0,1,\ldots,r-1\}$. Further, this leads to 
\begin{equation*}
Q_{E}\left( n\right) =\left\{
\begin{array}{ll}
\frac{r^{r+1}}{(r+1)^r}, & n\equiv 0\pmod{r}, \\
1, & n\equiv 1,\ldots ,r-2\pmod{r}, \text{ and} \\
\frac{(r+1)^{r}}{r^{r+1}}, & n\equiv r-1\pmod{r}.
\end{array}
\right.
\end{equation*}
Therefore, Theorem \ref{thm: basic} asserts that $p_{E,\ell }(n)$ is log-concave if $n\equiv0\pmod{r}$ and log-convex if $n\equiv -1\pmod{r}$ for all but finitely many numbers $\ell$.
This finishes the proof of
Theorem~\ref{th1rr+1} in the cases
$n\equiv 0,r-1\pmod{r}$.
On the other hand, the cases of $n\not\equiv 0,-1 \pmod{r}$ need some additional investigation.

\section{The problematic quotient case: $Q_{E}(n)=1$}\label{Section 5}

In the previous section, we got a few problematic cases with $Q_{E
}\left( n\right) =1$ (\ref{subsection S=(1,2,3)}, \ref{subsection S=(1,3)}, \ref{subsection S=(1,r,r+1)}). Since it is not clear what kind of log-behavior one should predict in these instances, we slightly extend the approach from Section~\ref{Preliminaries}
to obtain a
more general criterion than Theorem \ref{thm: basic}. In order to do that, an additional assumption throughout this section is needed, namely, we require that 
\begin{align}\label{key: estimate}
M_E(n)>M_E(n-1)
\end{align}
for all $n\geq n_0$, where $n_0$ is some positive integer.   

\begin{remark}
Let us notice that the inequality (\ref{key: estimate}) is satisfied in both
cases \ref{subsection S=(1,2,3)}
and \ref{subsection S=(1,r,r+1)}.
However, it is easy to notice that this condition does not hold in the
case \ref{subsection S=(1,3)}---there we have $M_E(n)=M_E(n-1)$ whenever
$n\geq4$ and $n\not\equiv0\pmod{3}$.
\end{remark}

We start by determining more sophisticated estimations for both the lower and the upper bounds of $p_{E,\ell }(n)$. For the upper bound, we have that

\begin{eqnarray*}
p_{E,\ell }(n) & =  & 
\sum _{k\leq n}\sum _{\substack{m_{1},\ldots ,m_{k} \geq 1 \\ m_{1}+\ldots +m_{k}=n}}
\frac{1}{k!} \, \frac{g_{E,\ell }\left( m_{1}\right) 
\cdots g_{E,\ell }\left( m_{k}\right) }{m_{1}\cdots m_{k}}\\
&\leq &
\sum _{k\leq n}\sum _{\substack{ m_1,\ldots,m_k\in S \\ m_{1}+\ldots +m_{k}=n  \\ m_1\cdots m_k=M_E}}
\frac{1}{k!} \, 
\frac{g_{E,\ell }\left( m_{1}\right) 
\cdots g_{E,\ell }\left( m_{k}\right) }{m_{1}\cdots m_{k}}\\
&
&{}+
\sum _{k\leq n}\sum _{\substack{m_1,\ldots,m_k\in S \\ m_{1}+\ldots +m_{k}=n \\ m_1\cdots m_k\not=M_E}}
\frac{1}{k!} \, 
\frac{\sigma_{E,1}\left( m_{1}\right) 
\cdots \sigma_{E,1}\left( m_{k}\right) }{m_{1}\cdots m_{k}}
\, \left(m_1 \cdots m_k\right)^{\Psi(\ell)}\\
&
&{}+
\sum _{k\leq n}\sum _{\substack{m_{1},\ldots ,m_{k} \geq 1 \\ m_{1}+\ldots +m_{k}=n \\ \exists 1\leq i\leq k: m_i\not\in S}}
\frac{1}{k!} \, 
\frac{\sigma_{E,1}\left( m_{1}\right) 
\cdots \sigma_{E,1}\left( m_{k}\right) }{m_1 \cdots m_k}
\, \left(m_{1,S} \cdots m_{k,S}\right)^{\Psi(\ell)}.
\end{eqnarray*}
Now, let us observe that the expression in the last line might be estimated from above by
\begin{align*}
&\sum _{k\leq n}\sum _{\substack{m_{1},\ldots ,m_{k} \geq 1 \\ m_{1}+\ldots +m_{k}=n \\ \exists 1\leq i\leq k: m_i\not\in S}}
\frac{1}{k!} \, 
\frac{\sigma_{E,1}\left( m_{1}\right) 
\cdots \sigma_{E,1}\left( m_{k}\right) }{m_1 \cdots m_k}
\, \left(\max_{1\leq l\leq n-1}\max_{\substack{ n_1,\ldots,n_l\in S \\ n_{1}+\ldots +n_{l}\leq n-1}}\prod_{i=1}^ln_i\right)^{\Psi(\ell)}\\
= &\sum _{k\leq n}\sum _{\substack{m_{1},\ldots ,m_{k} \geq 1 \\ m_{1}+\ldots +m_{k}=n \\ \exists 1\leq i\leq k: m_i\not\in S}}
\frac{1}{k!} \, 
\frac{\sigma_{E,1}\left( m_{1}\right) 
\cdots \sigma_{E,1}\left( m_{k}\right) }{m_1 \cdots m_k}
\, M_E(n-1)^{\Psi(\ell)}\\
\leq &\sum _{k\leq n}\sum _{\substack{m_{1},\ldots ,m_{k} \geq 1 \\ m_{1}+\ldots +m_{k}=n \\ \exists 1\leq i\leq k: m_i\not\in S}}
\frac{1}{k!} \, 
\frac{\sigma_{E,1}\left( m_{1}\right) 
\cdots \sigma_{E,1}\left( m_{k}\right) }{m_1 \cdots m_k}
\, \Tilde{M}_E(n)^{\Psi(\ell)}
\end{align*}
 where
 \begin{equation*}\label{max:Tilde{M_E}}
\Tilde{M}_{E}\left( n\right) :=\max _{k\leq n}
\max _{\substack{m_1,\ldots,m_k \in S \\ m_{1}+\ldots +m_{k}=n \\ m_1\cdots m_k\not=M_E}}m_{1}\cdots m_{k},
\end{equation*}
and the last inequality is a consequence of (\ref{key: estimate}) and the fact that we can always put $1$ as an additional part to any composition of $n-1$. Thus, we get that
\begin{eqnarray*}
p_{E,\ell }(n) &\leq &
\sum _{k\leq n}\sum _{\substack{ m_1,\ldots,m_k\in S \\ m_{1}+\ldots +m_{k}=n  \\ m_1\cdots m_k=M_E}}
\frac{1}{k!} \, 
\frac{g_{E,\ell }\left( m_{1}\right) 
\cdots g_{E,\ell }\left( m_{k}\right) }{m_{1}\cdots m_{k}}\\
&
&{}+
\sum _{k\leq n}\sum _{\substack{m_1,\ldots,m_k\in S \\ m_{1}+\ldots +m_{k}=n \\ m_1\cdots m_k\not=M_E}}
\frac{1}{k!} \, 
\frac{\sigma_{E,1}\left( m_{1}\right) 
\cdots \sigma_{E,1}\left( m_{k}\right) }{m_{1}\cdots m_{k}}
\, \left(m_1 \cdots m_k\right)^{\Psi(\ell)}\\
&
&{}+
\sum _{k\leq n}\sum _{\substack{m_{1},\ldots ,m_{k} \geq 1 \\ m_{1}+\ldots +m_{k}=n \\ \exists 1\leq i\leq k: m_i\not\in S}}
\frac{1}{k!} \, 
\frac{\sigma_{E,1}\left( m_{1}\right) 
\cdots \sigma_{E,1}\left( m_{k}\right) }{m_1 \cdots m_k}
\, \Tilde{M}_E(n)^{\Psi(\ell)}\\
&\leq &
\sum _{k\leq n}\sum _{\substack{ m_1,\ldots,m_k\in S \\ m_{1}+\ldots +m_{k}=n  \\ m_1\cdots m_k=M_E}}
\frac{1}{k!} \, 
\frac{g_{E,\ell }\left( m_{1}\right) 
\cdots g_{E,\ell }\left( m_{k}\right) }{m_{1}\cdots m_{k}}+p_{E,1}(n)\Tilde{M}_E(n)^{\Psi(\ell)}.
\end{eqnarray*}
The lower bound takes the form
\begin{eqnarray*}
p_{E,\ell }(n) & \geq  & \sum _{k\leq n}\sum _{\substack{ m_1,\ldots,m_k\in S \\ m_{1}+\ldots +m_{k}=n  \\ m_1\cdots m_k=M_E}}
\frac{1}{k!} \, 
\frac{g_{E,\ell }\left( m_{1}\right) 
\cdots g_{E,\ell }\left( m_{k}\right) }{m_{1}\cdots m_{k}}.
\end{eqnarray*}

If the maximal product $M_E(n)$ is
attained by a unique $S$-partition, say $(m_1,\ldots,m_k)$, then one can
rewrite the above estimations as follows: 
\begin{align}\label{def: A_E(n)}
A_E(n)\prod_{i=1}^k\frac{g_{E,\ell }(m_i)}{m_i}\leq p_{E,\ell }(n) \leq A_E(n)\prod_{i=1}^k\frac{g_{E,\ell }(m_i)}{m_i}+p_{E,1}(n)\Tilde{M}_E(n)^{\Psi(\ell)},
\end{align}
where $A_E(n)$ is the number of all compositions of $n$ with parts in $S$ for
which the maximal value $M_E(n)$ is attained divided by $k!$.

Now, we are ready to present an extension of Theorem \ref{thm: basic}.
\begin{theorem}\label{thm: involved}
Let $E$, $\Phi$, $\Psi$, and
$\{f_{\ell}\}_{\ell}$ be given. Suppose further that the inequality
$M_E(n)>M_E(n-1)$ holds for every $n\geq n_0$ and some positive integer
$n_0$. Let us also assume that $Q_{E}\left( n\right) =1$ for some
$n>n_0$, and  the maximal products $M_E(n), M_E(n-1)$,
and $M_E(n+1)$ are attained by the unique $S$-partitions $(x_1,\ldots,x_a)$,
$(y_1,\ldots,y_b)$, and $(z_1,\ldots,z_c)$ (respectively) such that
\begin{align*}
\frac{g_{\ell}(x_1)^2\cdots g_{\ell}(x_a)^2}{x_1^2\cdots x_a^2}\sim
\frac{g_{\ell}(y_1)\cdots g_{\ell}(y_b)g_{\ell}(z_1)\cdots g_{\ell}(z_c)}{y_1\cdots y_bz_1\cdots z_c}
\end{align*}
for large $\ell $.
Then $p_{E,\ell }(n)$ is strictly
\begin{enumerate}
\item  log-concave at $n$ for all sufficiently large values of $\ell$ if $\frac{A_E(n)^2}{A_E(n-1)A_E(n+1)}>1$,
\item  log-convex at $n$ for all sufficiently large values of $\ell$ if $\frac{A_E(n)^2}{A_E(n-1)A_E(n+1)}<1$,
\end{enumerate}
where the values $A_E(n), A_E(n-1)$, and $A_E(n+1)$ are defined as in  (\ref{def: A_E(n)}).
\end{theorem}
\begin{proof}
Since both cases are very similar, let us only deal with the criterion for the log-concavity property. The assumptions from the statement together with the inequalities (\ref{def: A_E(n)}) assert that
\begin{equation*}
p_{E,\ell }(n)^2\geq A_E(n)^2\left(\prod_{i=1}^a\frac{g_{E,\ell }(x_i)}{x_i}\right)^2
\end{equation*}
and
\begin{align*}
p_{E,\ell }(n-1) \, p_{E,\ell }(n+1)\leq &\left(A_E(n-1)\prod_{i=1}^b\frac{g_{E,\ell }(y_i)}{y_i}+p_{E,1}(n-1)\Tilde{M}_E(n-1)^{\Psi(\ell)}\right)\\
&{}\cdot
\left(A_E(n+1)\prod_{i=1}^c\frac{g_{E,\ell }(z_i)}{z_i}+p_{E,1}(n+1)\Tilde{M}_E(n+1)^{\Psi(\ell)}\right)\\
<& A_E(n-1)A_E(n+1)\prod_{i=1}^b\frac{g_{E,\ell }(y_i)}{y_i}\prod_{i=1}^c\frac{g_{E,\ell }(z_i)}{z_i}\\
&{}+
A_E(n-1)p_{E,1}(n+1)M_E(n-1)^{\Psi(\ell)+1}\Tilde{M}_E(n+1)^{\Psi(\ell)}\\
&{}+
A_E(n+1)p_{E,1}(n-1)M_E(n+1)^{\Psi(\ell)+1}\Tilde{M}_E(n-1)^{\Psi(\ell)}\\
&{}+
p_{E,1}(n-1)p_{E,1}(n+1)\left(\Tilde{M}_E(n-1)\Tilde{M}_E(n+1)\right)^{\Psi(\ell)}.
\end{align*}
Next, we want to compare all of the above four summands with the estimations for $p_{\ell,S}(n)$.
In the first case, we obtain that
\begin{align*}
\frac{A_E(n-1)A_E(n+1)\prod_{i=1}^b\frac{g_{E,\ell }(y_i)}{y_i}\prod_{i=1}^c\frac{g_{E,\ell }(z_i)}{z_i}}{A_E(n)^2\left(\prod_{i=1}^a\frac{g_{E,\ell }(x_i)}{x_i}\right)^2}
\sim \frac{A_E(n-1)A_E(n+1)}{A_E(n)^2}<1
\end{align*}
for sufficiently large values of
$\ell $.
The second one, on the other hand, requires a little bit more computations:
\begin{align*}
&\frac{A_E(n-1)p_{E,1}(n+1)M_E(n-1)^{\Psi(\ell)+1}\Tilde{M}_E(n+1)^{\Psi(\ell)}}{A_E(n)^2\left(\prod_{i=1}^a\frac{g_{E,\ell }(x_i)}{x_i}\right)^2}\\
\leq
&2\frac{A_E(n-1)p_{E,1}(n+1)M_E(n-1)^{\Psi(\ell)+1}\Tilde{M}_E(n+1)^{\Psi(\ell)}}{A_E(n)^2\prod_{i=1}^b\frac{g_{E,\ell }(y_i)}{y_i}\prod_{i=1}^c\frac{g_{E,\ell }(z_i)}{z_i}}\\
\leq& 2\frac{A_E(n-1)p_{E,1}(n+1)M_E(n-1)^{\Psi(\ell)+1}\Tilde{M}_E(n+1)^{\Psi(\ell)}}{A_E(n)^2\left(M_E(n-1)M_E(n+1)\right)^{\Phi(\ell)}}\\
\leq
&2\frac{A_E(n-1)p_{E,1}(n+1)M_E(n-1)^{\Psi(\ell)+1}\Tilde{M}_E(n+1)^{\Psi(\ell)}}{A_E(n)^2\left(M_E(n-1)M_E(n+1)\right)^{-
B}\left(M_E(n-1)M_E(n+1)\right)^{\Psi(\ell)}}\\
=&C_1(n)\left(\frac{\Tilde{M}_E(n+1)}{M_E(n+1)}\right)^{\Psi(\ell)}\rightarrow 0
\end{align*}
for sufficiently large $\ell $.
The third case is similar to the second one and presents as follows:
\begin{align*}
&\frac{A_E(n+1)p_{E,1}(n-1)M_E(n+1)^{\Psi(\ell)+1}\Tilde{M}_E(n-1)^{\Psi(\ell)}}{A_E(n)^2\left(\prod_{i=1}^a\frac{g_{E,\ell }(x_i)}{x_i}\right)^2}\\
\leq
&2\frac{A_E(n+1)p_{E,1}(n-1)M_E(n+1)^{\Psi(\ell)+1}\Tilde{M}_E(n-1)^{\Psi(\ell)}}{A_E(n)^2\prod_{i=1}^b\frac{g_{E,\ell }(y_i)}{y_i}\prod_{i=1}^c\frac{g_{E,\ell }(z_i)}{z_i}}\\
\leq& 2\frac{A_E(n+1)p_{E,1}(n-1)M_E(n+1)^{\Psi(\ell)+1}\Tilde{M}_E(n-1)^{\Psi(\ell)}}{A_E(n)^2\left(M_E(n-1)M_E(n+1)\right)^{\Phi(\ell)}}\\
\leq
&2\frac{A_E(n+1)p_{E,1}(n-1)M_E(n+1)^{\Psi(\ell)+1}\Tilde{M}_E(n-1)^{\Psi(\ell)}}{A_E(n)^2\left(M_E(n-1)M_E(n+1)\right)^{-
B}\left(M_E(n-1)M_E(n+1)\right)^{\Psi(\ell)}}\\
=&C_2(n)\left(\frac{\Tilde{M}_E(n-1)}{M_E(n-1)}\right)^{\Psi(\ell)}\rightarrow 0
\end{align*}
for sufficiently large $\ell $.
The last instance can be estimated in the following way
\begin{align*}
&\frac{p_{E,1}(n-1)p_{E,1}(n+1)\left(\Tilde{M}_E(n-1)\Tilde{M}_E(n+1)\right)^{\Psi(\ell)}}{A_E(n)^2\left(\prod_{i=1}^a\frac{g_{E,\ell }(x_i)}{x_i}\right)^2}\\
\leq
&2\frac{p_{E,1}(n-1)p_{E,1}(n+1)\left(\Tilde{M}_E(n-1)\Tilde{M}_E(n+1)\right)^{\Psi(\ell)}}{A_E(n)^2\prod_{i=1}^b\frac{g_{E,\ell }(y_i)}{y_i}\prod_{i=1}^c\frac{g_{E,\ell }(z_i)}{z_i}}\\
\leq& 2\frac{p_{E,1}(n-1)p_{E,1}(n+1)\left(\Tilde{M}_E(n-1)\Tilde{M}_E(n+1)\right)^{\Psi(\ell)}}{A_E(n)^2\left(M_E(n-1)M_E(n+1)\right)^{\Phi(\ell)}}\\
\leq
&2\frac{p_{E,1}(n-1)p_{E,1}(n+1)\left(\Tilde{M}_E(n-1)\Tilde{M}_E(n+1)\right)^{\Psi(\ell)}}{A_E(n)^2\left(M_E(n-1)M_E(n+1)\right)^{-
B}\left(M_E(n-1)M_E(n+1)\right)^{\Psi(\ell)}}\\
=&C_3(n)\left(\frac{\Tilde{M}_E(n-1)\Tilde{M}_E(n+1)}{M_E(n-1)M_E(n+1)}\right)^{\Psi(\ell)}\rightarrow 0
\end{align*}
for sufficiently large values of
$\ell $.
All of the above calculations lead us finally to the conclusion that
\begin{align*}
\frac{p_{E,\ell }(n-1)p_{E,\ell }(n+1)}{p_{E,\ell }(n)^2}\rightarrow \frac{A_E(n-1)A_E(n+1)}{A_E(n)^2}<1,
\end{align*}
for sufficiently large values of
$\ell $
and this finishes the proof.
\end{proof}

\begin{remark}
Once again, the natural question arises from the statement of Theorem \ref{thm: involved}, namely, what can one say about the log-behavior of $p_{E,\ell }(n)$ if $\frac{A_E(n)^2}{A_E(n-1)A_E(n+1)}=1$. That case becomes even more difficult, and we saw
in Example~\ref{example13} that there is no
easy answer to that problem.
\end{remark}

Now, we apply Theorem \ref{thm: involved} to investigate some of the remaining cases from Section $4$. At this point, let us observe that we can not simply use our new criterion to deal with the instance \ref{subsection S=(1,2,3)}. Mainly because, we do not have the property with unique partition for large $n\equiv1\pmod{3}$ whenever $4\not\in E$---that directly follows from Lemma \ref{Lemma: a_2=2}. Therefore, we divide the case \ref{subsection S=(1,2,3)} into two sub-cases depending on belonging of the number $4$ to the exception set.

\subsection{The case of $\mathbf{1,2,3 \not\in E}$ and $\mathbf{4\in E}$}\label{Paragraph: 1,2,3 not in E 4 in E}
Here, Lemma \ref{Lemma: a_2=2} guarantees that the assumption from the statement of Theorem \ref{thm: involved} hold, and one can apply it for any $n>3$ with $n\equiv2\pmod{3}$. It follows that
\begin{align*}
\frac{A_E(n)^2}{A_E(n-1)A_E(n+1)}=\frac{\frac{1}{\left(\frac{n-2}{3}\right)!^2}}{\frac{1}{2!\left(\frac{n-5}{3}\right)!\left(\frac{n+1}{3}\right)!}}=\frac{2(n+1)}{n-2}>1.
\end{align*}
Hence, Theorem \ref{thm: involved} points out that $p_{E,\ell }(n)$ is log-concave for all but finitely many numbers $\ell$.
This finishes the proof of
Theorem~\ref{th1234} in the case
$4\in E$.

\subsection{The case of $\mathbf{1,2,3,4 \not\in E}$} In order to use Theorem \ref{thm: involved} to deal with that instance, we need to consider two additional cases depending on the relation between $g_\ell(2)$ and $g_\ell(4)$ for all sufficiently large values of $\ell$.

\subsubsection{}

Let us assume that there exists $\ell_0>0$ such that $2g_\ell(4)>\delta g_\ell(2)^{2}
$ for some fixed $\delta>1$ and every $\ell\geq\ell_0$. Moreover, let $n\equiv2\pmod{3}$ and $n>2(\delta+2)/(\delta-1)$. Here, all of the conditions from the statement of Theorem \ref{thm: involved} are satisfied except the uniqueness of the
$S$-partition for which $M_E(n-1)$ is attained (as Lemma \ref{Lemma: a_2=2} indicates). But it does not make a big difference. In fact, we just need to replace the value of $A_E(n-1)$ by the number $A_E^\star(n-1)$ that takes into account all of the compositions
with parts in $S$, for which $M_E(n-1)$ is attained. We have that
\begin{align*}
p_{E,\ell }(n-1) &\geq  \sum _{k\leq n-1}\sum _{\substack{ m_1,\ldots,m_k\in S \\ m_{1}+\ldots +m_{k}=n-1  \\ m_1\cdots m_k=M_E}}
\frac{1}{k!} \, 
\frac{g_{E,\ell }\left( m_{1}\right) 
\cdots g_{E,\ell }\left( m_{k}\right) }{m_{1}\cdots m_{k}}\\
&=\frac{1}{2!\left(\frac{n-5}{3}\right)!}\frac{g_\ell(2)^2}{2^2}\left(\frac{g_\ell(3)}{3}\right)^{\frac{n-5}{3}}+\frac{1}{\left(\frac{n-5}{3}\right)!}\frac{g_\ell(4)}{4}\left(\frac{g_\ell(3)}{3}\right)^{\frac{n-5}{3}}\\
&>\frac{1+\delta}{2!\left(\frac{n-5}{3}\right)!}\frac{g_\ell(2)^2}{2^2}\left(\frac{g_\ell(3)}{3}\right)^{\frac{n-5}{3}}=:A_E^\star(n-1)
\frac{\left(
g_{\ell }\left( 2\right)
\right) ^{2}}{2^{2}}\left(
\frac{g_{\ell }\left( 3\right)
}{3}\right) ^{\frac{n-5}{3}},
\end{align*}
where the last inequality holds for all $\ell\geq\ell_0$. Now, we can use
Theorem \ref{thm: involved} for the values $A_E(n)$, $A_E(n-1)^\star$,
and $A_E(n+1)$ and obtain that
\begin{align*}
\frac{A_E(n)^2}{A_E^\star(n-1)A_E(n+1)}\sim\frac{2}{1+\delta}\frac{n+1}{n-
2}<1,
\end{align*}
where the last inequality follows from our assumption that $n>2(\delta+2)/(\delta-1)$. In conclusion, we get that $p_{E,\ell }\left( n\right)
$ is log-convex for all but finitely many numbers $\ell$.

\subsubsection{}

Let us assume that $n>3$, $n\equiv2\pmod{3}$, and there exists $\ell_0>0$ such that $2g_\ell(4)<\delta
g_{\ell }\left( 2\right)^{2}
$ for some fixed $\delta<1$ and every $\ell\geq\ell_0$. Since the reasoning is analogous to the previous one, we omit it here and only notice that
\begin{align*}
\frac{A_E(n)^2}{A_E^\star(n-1)A_E(n+1)}\sim\frac{2}{1+\delta}\frac{n+1}{n-2}>1,
\end{align*}
which indicates that $p_{E,\ell }\left( n\right)
$ is log-concave for all but finitely many numbers $\ell$.

At this point, it is worth saying that our approach does not cover the cases investigated by Abdesselam, Heim, and Neuhauser \cite{AHN24, HN22}, because the appropriate parameters $\delta$ and $\ell_0$ do not exist there.

\subsection{The case of $\mathbf{s\geq3}$, $\mathbf{1,s,s+1 \not\in E}$ and $\mathbf{2,3,\ldots,s-1 \in E}$}\label{problematic subsection S=(1,r,r+1)}

Here, Lemma \ref{Lemma: a_3=a_2+1} maintains that one can directly apply Theorem \ref{thm: involved} for any  $n>r(r-1)(3r-1)/2+1$, $n\equiv j\pmod{r} $ with $j\in\{1,\ldots,r-2\}$, where $r=s$. We get that

\begin{align*}
\frac{\left( A_{E}\left( n\right) \right) ^{2}
}{A_E(n-1)A_E(n+1)}=\frac{j+1}{j}\frac{n-j(r+1)+r}{n-j(r+1)}>1,
\end{align*}
which means that $p_{E,\ell }(n)
$ is log-concave for all but finitely many numbers $\ell$.

At the end of this section, it is worth pointing out that the approach presented here can not be explicitly applied to deal with the last remaining case from Section \ref{subsection S=(1,3)}.

\section{Lemmata}

Here, we collect all auxiliary properties that are used throughout the manuscript.

\begin{lemma}\label{Lemma: a_2>2 a_j>2a_2}
Let $S=\{1,a_2,a_3,\ldots\}$ be such a set of positive integers that $1<a_2<a_3<\cdots$, $a_2\geq3$, $a_{j-1} < 2a_2$ and $a_j\geq 2a_2$ for some $j\geq3$. Then the largest value
\begin{align*}
\max_{1\leq k\leq n}\max_{\substack{ m_1,\ldots,m_k\in S \\ m_{1}+\ldots +m_{k}=n}}\prod_{i=1}^km_i
\end{align*}
can not be attained at the partition $(m_1,\ldots,m_k)$ if there exists $i\in\{1,2,\ldots,k\}$ such that $m_i=a_l$ for some $l\geq j$.
\end{lemma}
\begin{proof}
Let us suppose that $a_l$ occurs as a part of $(m_1,\ldots,m_k)$. Thus, we
may write $a_l=ua_2+v$ for some non-negative integers $u\geq2$ and
$v\in\{0,1,\ldots,a_2-1\}$. In particular, it follows that one can replace
the part $a_l$ by $u$ $a_2$s and $v$ ones. After such a switch it suffices to show that 
\begin{align*}
(u+1)a_2\leq a_2^u,
\end{align*}
and that is an elementary exercise. This completes the proof.
\end{proof}

As a direct consequence of the preceeding lemma, we get the following property.

\begin{corollary}\label{Corollary: a_2>2 a_3>2a_2}
Let $S=\{1,a_2,a_3,\ldots\}$ be such a set of positive integers that $1<a_2<a_3<\cdots$, $a_2\geq3$ and $2a_2\leq a_3$. Then we have that
\begin{equation*}
\max_{1\leq k\leq n}\max_{\substack{ m_1,\ldots,m_k\in S \\ m_{1}+\ldots +m_{k}=n}}\prod_{i=1}^km_i=a_2^{\floor{\frac{n}{a_2}}},
\end{equation*}
which is attained at the partition
\begin{equation*}
(a_2^{\floor{\frac{n}{a_2}}},1^{n-a_2\floor{\frac{n}{a_2}}}).
\end{equation*}
\end{corollary}

Despite the fact that Corollary \ref{Corollary: a_2>2 a_3>2a_2} might be
applied for various sets $S$, we still do not have any information about the
considered maximum value, when $
a_2<a_3<2a_2$ or $a_2=2$. Let us investigate these cases separately. At first, we assume that $a_2=2$. In such a setting, one can deduce the following.

\begin{lemma}\label{Lemma: a_2=2}
Let $S=\{1,2,a_3,\ldots\}$ be such a set of positive integers that $2<a_3<a_4<\cdots$. For a natural number $n$, we have the following description of
\begin{align*}
\max_{1\leq k\leq n}\max_{\substack{ m_1,\ldots,m_k\in S \\ m_{1}+\ldots +m_{k}=n}}\prod_{i=1}^km_i.
\end{align*}
\begin{enumerate}

\item If $a_3=3$ and $4\not\in S$, then it is attained at
\begin{align*}
(1), &\text{ if } n=1,\\
(3^\frac{n}{3}), &\text{ if } n\equiv0\pmod{3},\\
 (3^\frac{n-4}{3},2,2), &\text{ if } n\geq4 \text{ and } n\equiv1\pmod{3},\\
 (3^\frac{n-2}{3},2), &\text{ if }n\equiv2\pmod{3}.
\end{align*}
Similarly, if $a_3=3$ and $a_4=4$, then the maximum value is attained at 
\begin{align*}
(1), &\text{ if } n=1,\\
(3^\frac{n}{3}), &\text{ if } n\equiv0\pmod{3},\\
 (3^\frac{n-4}{3},2,2),\left( 4,3^\frac{n-4}{3}\right)
, &\text{ if } n\geq4 \text{ and } n\equiv1\pmod{3},\\
 (3^\frac{n-2}{3},2), &\text{ if }n\equiv2\pmod{3}.
\end{align*}

\item If $a_3=4$ and $5\not\in S$, then it is attained at
\begin{align*}
(2^{\frac{n}{2}}),(4,2^{\frac{n-4}{2}}),\ldots,(4^{\frac{n}{4}}), &\text{ if } n\equiv0\pmod{4},\\
(2^{\frac{n-1}{2}},1),(4,2^{\frac{n-5}{2}},1),\ldots,(4^{\frac{n-1}{4}},1), &\text{ if } n\equiv1\pmod{4},\\
(2^{\frac{n}{2}}),(4,2^{\frac{n-4}{2}}),\ldots,(4^{\frac{n-2}{4}},2), &\text{ if } n\equiv2\pmod{4},\\
(2^{\frac{n-1}{2}},1),(4,2^{\frac{n-5}{2}},1),\ldots,(4^{\frac{n-3}{4}},2,1), &\text{ if } n\equiv3\pmod{4}.
\end{align*} 
On the other hand, if $a_3=4$ and $a_4=5$, then the relevant partitions take the forms $(1), (2)$ and
$(2,1)$ for $n\leq3$, and
\begin{align*}
(2^{\frac{n}{2}}),(4,2^{\frac{n-4}{2}}),\ldots,(4^{\frac{n}{4}}), &\text{ if } n\equiv0\pmod{4},\\
(5,2^{\frac{n-5}{2}}),(5,4,2^{\frac{n-9}{2}}),\ldots,(5,4^{\frac{n-5}{4}}), &\text{ if } n\equiv1\pmod{4},\\
(2^{\frac{n}{2}}),(4,2^{\frac{n-4}{2}}),\ldots,(4^{\frac{n-2}{4}},2), &\text{ if } n\equiv2\pmod{4},\\
(5,2^{\frac{n-5}{2}}),(5,4,2^{\frac{n-9}{2}}),\ldots,(5,4^{\frac{n-7}{4}},2), &\text{ if } n\equiv3\pmod{4},
\end{align*}
for $n\geq4$.

\item If $a_3
=5$, then it is attained at 
\begin{align*}
(1), &\text{ if } n=1,\\
(2,1), &\text{ if } n=3,\\
(5,2^{\frac{n-5}{2}}), &\text{ if } n\geq 5\text{ and }n\equiv1\pmod{2},\\
 (2^{\frac{n}{2}}), &\text{ if } n\equiv0\pmod{2}.
\end{align*}

\item If $a_3\geq 6$, then it is attained at
\begin{align*}
(2^{\floor{\frac{n}{2}}},1^{n-2\floor{\frac{n}{2}}}).
\end{align*}
\end{enumerate}
\end{lemma}

\begin{proof}
The proof of the point $(4)$ is an analogue
of the proof of Lemma \ref{Lemma: a_2>2 a_j>2a_2}. Therefore, we omit it here. On the other hand, the remaining cases are similar to each other. Hence, we show $(2)$ and leave the rest as an exercise to the reader.

At first, let us suppose that $5\notin S$. Since we can always replace $4$ by
two $2$s and vice versa, $1$ can not occur more than once as a part, and we know that the point $(4)$ is true, it follows that the desired maximum is attained at the appropriate partitions.

On the other hand, if $a_4=5$, then the number $5$ can occur at most once as
a part. Otherwise, it is better to replace two $5$s by, for example, five
$2$s ($5^2<2^5$). Additionally, it is easy to see that $1$ can not appear as a part if $n\geq5$. Indeed, we may just simply replace $(2,2,1)$ or $(4,1)$ by $(5)$, or $(5,1)$ by $(2,2,2)$. This ends the proof.
\end{proof}

\begin{lemma}\label{Lemma: a_2-1 is the max of occurrences}  
Let $S=\{1,a_2,a_3,\ldots\}$ be such a set of positive integers that $1<a_2<a_3<\cdots$, $a_2\geq3$ and $a_{j} < 2a_2$ and $a_{j+1}\geq2a_2$ for some $j\geq3$. For a natural number $n$, the value of
\begin{align*}
\max_{1\leq k\leq n}\max_{\substack{ m_1,\ldots,m_k\in S \\ m_{1}+\ldots +m_{k}=n}}\prod_{i=1}^km_i.
\end{align*}
can not be attained at the partition $(m_1,\ldots,m_k)$ if there exists $a_i$ for $i\in\{1,3,4,\ldots,j\}$ which occurs more than $a_2-1$ times as a part in $(m_1,\ldots,m_k)$.
\end{lemma}
\begin{proof}
Clearly, $1$ can not appear as a part more than $a_2-1$ times. Let us suppose
that there is some $a_i$ for $i\in\{3,4,\ldots,j\}$ with
multiplicity
$a_2$ in the partition. Thus, we may replace $a_2$ $a_i$s by $a_i$ $a_2$s and examine whether the inequality
\begin{align*}
a_i^{a_2}<a_2^{a_i}
\end{align*}
is true. For the sake of convenience, let us set $u:=a_2$ and $v:=a_i$, and notice that $u<v$. Hence, we may rewrite the above inequality as follows
\begin{align*}
\frac{\ln{v}}{v}<\frac{\ln{u}}{u}.
\end{align*} 
Finally, let us observe that the function $\ln{x}/x$ is decreasing for $x>e$. This ends the proof.
\end{proof}

\begin{lemma}\label{Lemma: a_3=a_2+1}
Let $S=\{1,a_2,a_3,\ldots\}$ be such a set of positive integers that $1<a_2<\cdots$, $a_2\geq3$, $a_3=a_2+1$, $a_{j}<2a_2$ and $a_{j+1}\geq2a_2$ for some $j\geq3$. Suppose further that $n\geq a_2(a_2-1)(3a_2-1)/2$. Then the value of 
\begin{align*}
\max_{1\leq k\leq n}\max_{\substack{ m_1,\ldots,m_k\in S \\ m_{1}+\ldots +m_{k}=n}}\prod_{i=1}^km_i.
\end{align*}
is attained at
\begin{align*}
(a_3^{i},a_2^{\frac{n-i\cdot a_3}{a_2}}), &\text{ if } n\equiv i\pmod{a_2}
\end{align*}
for every $i\in\{0,1,\ldots,a_2-1\}$.
\end{lemma}
\begin{proof}
At first, let us note that Lemma \ref{Lemma: a_2>2 a_j>2a_2} ensures that $a_i$ can not appear as a part of $n$ for any $i\geq j+1$. On the other hand, Lemma \ref{Lemma: a_2-1 is the max of occurrences} points out that each number $a_i$ for $i\in\{1,3,4\ldots,j\}$ might occur at most $a_2-1$ times as a part of $(m_1,\ldots,m_k)$. Since we have that
\begin{align*}
&1\cdot (a_2-1)+ a_2(a_2-2)+ (a_2-1)(a_3+\cdots+a_j)\\
<&a_2+a_2(a_2-2)+(a_2-1)(a_2+1+a_2+2+\cdots+2a_2-1)\\
= &a_2(a_2-1)(3a_2-1)/2\leq n,
\end{align*}
it is transparent that $a_2$ needs to occur at least $a_2-1$ times as a part of $(m_1,\ldots,m_k)$. Hence, let us suppose that there is some $a_i$ for $i\in\{1,4,5,\ldots,\}$, which appears as a part of $(m_1,\ldots,m_k)$. If $i=1$, then it suffices to replace $1$ and $a_2$ by $a_3$. On the other hand, if $i\geq 4$, then let us write $a_i=a_3+m$ for some $1\leq m \leq a_2-2$, and notice that one can swap $(a_i,a_3^s,a_2^t)$ by $(a_3^{s+m+1},a_2^{t-m})$ as $t>m$. Thus, it is enough to show that
\begin{align*}
(a_3+m)a_3^sa_2^t<a_3^{s+m+1}a_2^{t-m},
\end{align*} 
or, equivalently,
\begin{align*}
a_3+m<\left(\frac{a_3}{a_2}\right)^ma_3.
\end{align*}
In order to do that, we set 
$$\varphi (x):=\left(\frac{x+1}{x}\right)^m(x+1)-(x+1)-m=(x+1)\left(\left(\frac{x+1}{x}\right)^m-1\right)-m,$$
and observe that 
\begin{align*}
\varphi '(x)=\left(\frac{x+1}{x}\right)^m\left(\frac{x-m}{x}\right)-1.
\end{align*} 
One can also deduce that $\varphi '(x)$ is negative for $x>0$. To do that, we just examine the validity of
\begin{align*}
(x+1)^m(x-m)<x^{m+1},
\end{align*}
which might be rearranged as follows
\begin{align*}
\sum_{l=0}^{m-1}\binom{m}{l}x^{l+1}<m\sum_{l=0}^m\binom{m}{l}x^l.
\end{align*}
Hence, it suffices to check that there are relevant inequalities between corresponding coefficients standing next to $x^q$ and, at least, one of them is strict. That is an elementary exercise to verify. Thus, we get that $\varphi (x)$ is decreasing for $x>0$. Moreover, one can determine that
\begin{align*}
\lim_{x\to\infty}(x+1)\left(\left(\frac{x+1}{x}\right)^m-1\right)=\lim_{x\to\infty}\frac{mx^m+r(x)}{x^m}=m,
\end{align*} 
where $r(x)$ is a polynomial of degree $m-1$. Since we have that $\lim_{x\to\infty}\varphi (x)=0$ and $\varphi (x)$ is decreasing for $x>0$, the required property follows.
\end{proof}

\section*{Acknowledgments}
The first author was supported by the National Science Center grant no.\linebreak 2024/53/N/ST1/01538.


\begin{thebibliography}{AHN24}
\bibitem[AHN24]{AHN24} 
A. Abdesselam, B. Heim, and M. Neuhauser: \emph{On a mod $3$ property of $\ell $-tuples of pairwise commuting permutations\/}.
Ramanujan J., 66 no.\ 1
(2025), article 1, 15 p.

\bibitem[An98]{An98}
G. E. Andrews: The
Theory of
Partitions. Cambridge Univ.\ Press, Cambridge (1998).

\bibitem[B$^{3}$F24]{BBBF24}
W. Bridges, B. Brindle, K. Bringmann, and J. Franke:
\emph{Asymptotic expansion for partitions generated by infinite products.\/}
Math.\ Ann.,\ 390 no.\ 2
(2024), 2593--2632.
\texttt{https://doi.org/10.1007/s00208-024-02807-x}.

\bibitem[BFH24]{BFH24} K. Bringmann, J. Franke, and B. Heim:
\emph{Asymptotics of commuting $\ell$-tuples in
symmetric groups and log-concavity.\/}
Research in Number Theory, 10 no.\ 4 (2024), article 83, 19 p.
\texttt{https://doi.org/10.1007/s40993-024-00562-1}.






\bibitem[DT20]{DT20} G. Debruyne
and G. Tenenbaum:
\emph{The saddle-point method for general partition functions.\/} Indagationes
Math.,\ 31 (2020), 728--738.






\bibitem[GS06]{GS06} B. Granovsky and D. Stark: \emph{Asymptotic enumeration and logical limit laws for
expansive multisets.\/} J. London Math.\ Soc.\ (2), 73 (2006), 252--272.

\bibitem[GS12]{GS12} B. Granovsky and D. Stark: \emph{ 
A Meinardus theorem with multiple singularities.\/} Comm.\
Math.\ Phys., 314 (2012), 329--350.

\bibitem[GS15]{GS15} B. Granovsky and D. Stark:
\emph{Developments in the Khintchine--Meinardus
probabilistic
method for
asymptotic
enumeration.\/}
The Electronic Journal of Combinatorics,
22 no.\ 4
(2015),
article P4.32, 26 p.

\bibitem[GSE08]{GSE08}
B. Granovsky, D. Stark, and M. Erlihson:
\emph{Meinardus' theorem on weighted partitions:
Extensions and a probabilistic proof.\/}  Adv.\ Appl.\ Math.,\ 41 (2008),
307--328.


\bibitem[HN22]{HN22} B. Heim and M. Neuhauser:
\emph{Log-concavity of infinite product generating functions\/}. Res.\ Number
Theory, 8 no.\ 3 (2022), article 53.






\bibitem[Me54]{Me54} G. Meinardus:
\emph{Asymptotische Aussagen über Partitionen.\/}
Math.\ Z., 59 (1954), 388--398.





\end{thebibliography}
\end{document}